\documentclass{amsart}
\newcommand{\A}{{\mathfrak A}}

\newcommand{\Z}{{\mathbb Z}}

\newcommand{\Q}{{\mathbb Q}}
\theoremstyle{definition}

\theoremstyle{plain}
\newtheorem{Lemma}{\sc Lemma}
\newtheorem{Theorem}{\sc Theorem}
\newtheorem{Proposition}{\sc Proposition}
\newtheorem{Corollary}{\sc Corollary}
\newtheorem{Problem}{\sc Problem}
\theoremstyle{remark}
\newtheorem{Remark}{\sc Remark}
\newtheorem{Example}{\sc Example}

\begin{document}


\title{On abelian subgroups of finitely generated metabelian groups}

\author[V. H. Mikaelian, \,\,  A. Yu. Olshanskii]{
V. H. Mikaelian
\quad
A. Yu. Olshanskii
}

\dedicatory{To Professor Gilbert Baumslag to his 80'th birthday }

\begin{abstract}
In this note we introduce the class of $\mathcal H$-groups (or Hall groups) related to the class of $\mathcal B$-groups defined by Ph.~Hall in 1950's. Establishing some basic properties of Hall groups we use them to obtain results concerning embeddings of abelian groups.
In particular, we give an explicit classification of all abelian groups that can occur as subgroups in finitely-generated metabelian groups. Hall groups allow to give a  negative answer to the Baumslag's conjecture of 1990 on the cardinality of the set of isomorphism classes for abelian subgroups in finitely generated metabelian groups.
%
%Finally, generalizing a lemma of Neumanns, we for any $l$ show and example of a countable soluble group of length $l$ %that cannot be embedded into a finitely-generated soluble group of length $l+1$.
\end{abstract}

\date{8 February, 2013}

\maketitle

{\small \parskip0mm \tableofcontents }

%%%%%%%%%%%%%%%%%%%%%%%%%%%%%%%%%%%%%%%%%%%%%%%%%%%%%%%%%%%%%%%%%%
%%%%%%%%%%%%%%%%%%%%%%%%%%%%%%%%%%%%%%%%%%%%%%%%%%%%%%%%%%%%%%%%%%
%%%%%%%%%%%%%%%%%%%%%%%%%%%%%%%%%%%%%%%%%%%%%%%%%%%%%%%%%%%%%%%%%%
\section{Introduction}
\label{Introduction}

\noindent
Our note goes back to the old  paper of P. Hall~\cite{On the finiteness
of certain soluble groups} who obtained
 the properties of abelian {\it normal} subgroups in finitely
 generated metabelian and abelian-by-polycyclic groups. Let
 ${\mathcal B}$ be the class of all abelian groups $B$, where $B$ is
 an abelian normal subgroup of some finitely generated groups $G$ with
 polycyclic quotient $G/B$. It is proved in Lemmas 8 and 5.2 of~\cite{On the finiteness of certain soluble groups},  that $\mathcal B \subset \mathcal H,$
 where the class of countable abelian groups $\mathcal H$ can be defined
 as follows (in the present paper, we will say that the groups from $\mathcal H$
 are {\it Hall groups}). By definition, $H\in \mathcal H$ if

(1) $H$ is a (finite or) countable abelian group,

(2) $H=T\oplus K,$ where $T$ is a bounded torsion group (i.e.,
 the orders of all elements in $T$ are bounded), $K$ is  torsion free, and

(3) $K$ has a free abelian subgroup $F$ such that  $K/F$ is a torsion group with
 trivial $p$-subgroups for all primes except for the members of a finite set  $\pi=\pi(K)$.

Applying Ph.~Hall's results we observe the following.

\medskip

  \begin{Theorem}\label{main} A group $H$ is an abelian subgroup of a finitely generated
   abelian-by-polycyclic group if and only if  $H$ is a Hall group. Moreover, every Hall
   groups is embeddable into the derived subgroup of a 2-generated metabelian group.
   \end{Theorem}

  \medskip

Since by Baumslag - Remeslennikov theorem~\cite{Baumslag73, Remeslennikov73} every finitely generated metabelian group is embeddable into finitely presented metabelian group, we have

  \medskip

  \begin{Corollary} Let $H$ be an abelian group. The following properties are equivalent.

  (1) $H$ is a subgroup of a finitely generated metabelian group;

  (2) $H$ is a subgroup of finitely generated abelian-by-polycyclic group;

  (3) $H$ is a subgroup of finitely presented metabelian group;

  (4) $H$ is a subgroup of a $2$-generated metabelian group;

  (5) $H$ is a Hall group.
  \end{Corollary}

  Recall that G.~Baumslag, U.~Stammbach, and R.~Strebel~\cite{Baumslag88} proved that the
  2-generated free metabelian group contains uncountably many non-isomorphic
  subgroups.  On the other hand, in 1990, reflecting on Hall's results of \cite{On the finiteness of certain soluble groups},   Gilbert Baumslag \cite{Baumslag90} wrote: {\it ``It is easy to see that the abelian subgroups of a free
metabelian group are free abelian. So there are only a countable number
of isomorphism classes of abelian subgroups of a free metabelian group of
finite rank. It appears likely that this observation holds also for finitely
generated metabelian groups as a whole, but I have not yet managed to prove
this.''} Unfortunately this conjecture of Baumslag is unprovable because using Theorem \ref{main} we
obtain:

\begin{Theorem}
\label{cardin} There is a 2-generated metabelian group containing
continuously many pairwise non-isomorphic abelian subgroups.
\end{Theorem}

It follows from Theorem \ref{cardin} that the class of Hall groups $\mathcal H$ is bigger than
the class $\mathcal B$ since the set of non-isomorphic groups of $\mathcal B$ is countable. And we in fact  are able to explicitly build and example of a Hall group which is not a $\mathcal B$-group (see Section~\ref{Non-isomorphic subgroups of D(p)}).

V.A.~Roman'kov \cite{Romankov72, Romankov73} proved that every finitely generated nilpotent (every polycyclic) group is a subgroup
of a 2-generated nilipotent (resp. polycyclic) group.
In Sections~\ref{Embeddings of finitely generated abelian groups}, ~\ref{Further discussion} we present  two ``thrifty'' embeddings of finitely generated abelian groups into 2-generated metabelian, polycyclic   groups, and also examples of countable groups of solvable  length $l>1$ non-embeddable in finitely generated (l+1)-solvable groups, and two open questions.

\begin{Remark}
\label{Dn} Let $K$ be the subgroup from the above definition of Hall group.
   Since it is torsion free and countable, $K$ is embeddable into a
   countable direct power $\bigoplus{\Q}_i$ of the additive group of the rationals. Moreover,
   the basis elements of the subgroup $F$ can be mapped to the vectors $(0,\dots,0,1,0,\dots)$
   with one non-zero coordinate. Therefore $K$ becomes a subgroup of the group
   $\bigoplus_i (D_n)_i$, where $D_n$ consists of all rationals
   of the form $m/n^k$, with $n$ equal to the product of all primes from the set $\pi$.
   This observation and the property that the class $\mathcal H$ is closed under
   subgroups (Lemma 4.2 \cite{On the finiteness of certain soluble groups}) imply the equality ${\mathcal H}= \cup_{n=1}^{\infty} {\mathcal H}_n$, where ${\mathcal H}_n$ is the class
   of subgroups in the countable direct power of the group ${\Z}_n \oplus D_n$. (We use here
   that every bounded abelian group
   is embeddable in a direct sum of copies of ${\Z}_n$ for some $n$~\cite{Robinson}.)
   \end{Remark}

%%%%%%%%%%%%%%%%%%%%%%%%%%%%%%%%%%%%%%%%%%%%%%%%%%%%%%%%%%%%%%%%%%
%%%%%%%%%%%%%%%%%%%%%%%%%%%%%%%%%%%%%%%%%%%%%%%%%%%%%%%%%%%%%%%%%%
%%%%%%%%%%%%%%%%%%%%%%%%%%%%%%%%%%%%%%%%%%%%%%%%%%%%%%%%%%%%%%%%%%
\section{Necessity}
\label{Necessity}

\begin{Lemma}\label{nec}
   Let $H$ be an abelian subgroup of a finitely generated abelian-by-polycyclic group
   $G$. Then $H$ is a Hall group.
   \end{Lemma}

   \begin{proof} Let $B$ be an abelian normal subgroup of $G$ with polycyclic quotient $P=G/B$.
   By the inclusions $B\in\mathcal B \subset \mathcal H$,
   we have that the torsion subgroup $S$ of $B$ and $H\cap S$ are bounded groups.
   The quotient $H/(H\cap B) \cong HB/B$ is finitely generated being a subgroup of the polycyclic
   group $P,$ and so the torsion subgroup of it also is bounded. It follows that the
   torsion subgroup $T$ of $H$ is bounded as well, and so $H=T\oplus K$ by Kulikov's theorem~\cite{Fuchs, Robinson}.
    The torsion free group $K$ can be now regarded as a subgroup of the group $G/S$ with
    torsion free $C=B/S.$

    Since $P$  and its abelian subgroup $K/(K\cap C )$ are polycyclic, we have a series of subgroups $K_1\le K_2\le K$
    with torsion free $K_1=K\cap C$, free abelian $K_2/K_1$, and finite $K/K_2.$
   Since  $C\in {\mathcal B} \subset\mathcal H,$ we have that $K_1\in \mathcal H$ by Remark \ref{Dn},
   and moreover $K_1\in {\mathcal H}_n$ for some $n$.
   Since $K_2/K_1$ is free abelian, we have an isomorphism $K_2\cong K_1 \oplus (K_2/K_1)$,
    and therefore $K_2 \in {\mathcal H}_n$ holds. In turn, this implies
   that $K$ itself belongs to some ${\mathcal H}_m$ because it is torsion free, and the
   free subgroup $F$ of $K_2$ (from the definition of Hall group  applied to $K_2$) will work for $K$ as well
   (but       with a bigger finite set $\pi'$) since $K/K_2\le \infty$.

   Thus $H=T\oplus K$ with the required properties of $T$ and $K$. \end{proof}

   \begin{Remark} It is seen from the proof, that $H\in {\mathcal H}_k$, where $k$ depends only on $G$.
   %  $k=k(G)$.
   \end{Remark}

%%%%%%%%%%%%%%%%%%%%%%%%%%%%%%%%%%%%%%%%%%%%%%%%%%%%%%%%%%%%%%%%%%
%%%%%%%%%%%%%%%%%%%%%%%%%%%%%%%%%%%%%%%%%%%%%%%%%%%%%%%%%%%%%%%%%%
%%%%%%%%%%%%%%%%%%%%%%%%%%%%%%%%%%%%%%%%%%%%%%%%%%%%%%%%%%%%%%%%%%
\section{The embedding of $(D_n)^{\infty}$}
\label{The embedding of D_n}

Let $D(n)$ be the countable direct power of the group
   $D_n$ from Remark \ref{Dn}.

   \begin{Lemma}\label{D(s)}
The group $D(n)$  embeds into the derived subgroup of a $2$-generated metabelian group $G$.
In addition, $G/[G,G]$ is a free abelian group of rank 2.
\end{Lemma}

\begin{proof} Let us regard the elements of $D(n)$ as vectors ${\bf x}=(\dots,x_{-1},x_0,x_1,\dots)$
   with finitely many nonzero coordinates  $x_i=m_i/n^{k_i}$ ($i,m_i,k_i\in \Z, k_i\ge 0$).
We will use two automorphisms $\alpha$ and $\beta$ of the additive group $D(n)$, namely,
$\alpha$ shifts the coordinates, i.e., maps arbitrary vector $\bf x$ to $\bf y$, where
$y_i = x_{i-1}$ ($i\in \Z$), and $\beta$ multiplies all the vectors by $n$:
$\beta({\bf x})= n{\bf x}$.

 Since the automorphisms $\alpha$ and $\beta$ commute, we can build a split extension
 $S$ of $D(n)$ by a free abelian group $A=\langle a, b \rangle$   acting
 by conjugation  as follows: $a^{-1}{\bf x}a=\alpha({\bf x}), b^{-1}{\bf x}b=\beta({\bf x})$.
Since $b{\bf x} b^{-1}={\bf x}/n$ for any ${\bf x}\in D(n)$, we see that the normal closure in $S$ of the vector ${\bf c}=(\dots,0,1,0,\dots)$, where $1=c_0$, is the entire subgroup $D(n),$ and so $S=\langle a,b, {\bf c} \rangle $.

Now compute the commutator ${\bf d}= [b{\bf c}, a]$ using the additive notation for the elements from $D(n)$:
$${\bf d}=(-{\bf c})b^{-1}a^{-1}b{\bf c} a= (-{\bf c})a^{-1}{\bf c} a= (-{\bf c})+(a^{-1}{\bf c} a),$$
and so $d_0=-1, d_1=1$, and other coordinates of ${\bf d}$ are $0$.  Denote by $G$ the
metabelian subgroup generated by $a$ and $f=b{\bf c}$. It contains $\bf d$ in the derived subgroup $[G,G]$ and, for positive integers $k$-s, contains
 all the elements $(f)^k {\bf d} (f)^{-k} = b^{k} {\bf d} b^{-k}= {\bf d}/n^k$.  Therefore $[G,G]$
 contains a copy $D'$ of the group $D_n$, and all coordinates of the vectors $\bf z$ from $D'$
 are zero except for $z_0$ and $z_1$. It follows that the shifts $a^{-l}D'a^{l}$ with even
 exponents $l$-s generate in $[G,G]$, a subgroup isomorphic with $D(n)$.

 To prove the last statement of the lemma, we observe that the projection of the split extension $S$
 onto the free abelian group $A$ maps the generators $a$ and $f$ of the group $G$ to the free generators
 of $A$. It induces an epimorphism of the $2$-generated group $G/[G,G]$ onto $A$ that obviously
 must be an isomorphism.
    \end{proof}

%%%%%%%%%%%%%%%%%%%%%%%%%%%%%%%%%%%%%%%%%%%%%%%%%%%%%%%%%%%%%%%%%%
%%%%%%%%%%%%%%%%%%%%%%%%%%%%%%%%%%%%%%%%%%%%%%%%%%%%%%%%%%%%%%%%%%
%%%%%%%%%%%%%%%%%%%%%%%%%%%%%%%%%%%%%%%%%%%%%%%%%%%%%%%%%%%%%%%%%%
\section{Embeddings of arbitrary Hall groups}
\label{Embeddings of arbitrary Hall groups}

We denote by $Z(n)$ the countable direct
  power of a cyclic group of order $n$. Let $\A_m$ be the variety of abelian groups of exponents dividing the positive integer $m$ and let $A_m(X)$ be the verbal subgroup of a group $X$ corresponding to the variety $\A_m$,
i.e., $A_m(X)= [X,X]X^m$ ~\cite{HannaNeumann}.

\begin{Lemma}\label{wr} Let $F$ be a free group and $L$ a normal subgroup of infinite
index in $F$. Assume that the quotient $L/A_m(L)$ has an element $g$ of order $n\ge 1$. Then
the normal closure of this element in $F/A_m(L)$ has a subgroup isomorphic to $Z(n)$.
\end{Lemma}

\begin{proof} Since the variety $\A_m$ is abelian, the verbal $\A_m$-products of groups (see~\cite{HannaNeumann, Shmel'kin Wreath products and varieties}) coincide
with standard direct (restricted) wreath products. Therefore by Magnus - Shmelkin's
theorem~\cite{Shmel'kin Wreath products and varieties}, the groups $Y=F/A_m(L)$ embeds in the wreath product $V=B wr (F/L)$, where $B\cong F/A_m(F).$
Recall that $V$ is the semidirect product $CW$, where $C=F/L$ and
$W$ is the direct power of $B$ with the right regular action of the group $C$ by
conjugation on the set of direct factors. (The direct factors $B(c)$ of $W$ are isomorphic
to $B$ an indexed by the elements $c\in C$.) The embedding enjoys the properties:
$YW=V$ and (the image of) $L/A_m(L)$ becomes a
subgroup of $W$.

Thus, (under the embedding) the element $g$ belongs to a product $B(c_1)\times\dots\times B(c_t)$ for a finite set $\{c_1,\dots,c_t\}\subset C$.
Since $C$ is infinite, there is $d\in C$ such that the support $\{c_1d,\dots,c_td\}$ of the
conjugate elements $g_1=d^{-1}gd$ is disjoint with the set $\{c_1,\dots,c_t\}$. Recall that $d=yw$ for some $y\in Y,$ $w\in V$. Since $W$ is abelian, we have $g_1=ygy^{-1}$, and so $g_1$ belongs to the normal closure of
$g$ in $Y$, and it generates, together with $g$, the direct product of two cyclic groups of
order $n.$ Keep choosing elements $g_2, g_3,\dots$  in this manner, one can
obtain the elements of the normal closure of $g$ in $Y$, which generate  a subgroup  isomorphic to $Z(n)$.
\end{proof}

  \begin{Lemma} \label{Z(n) x D(n)}
  The abelian group $Z(n) \oplus D(n)$ embeds into the derived subgroup of a $2$-generator metabelian group $M$.
\end{Lemma}

\begin{proof} We may assume that $n>1$ (e.g., by Lemma \ref{D(s)}).

Let $F_2$ be a $2$-generated free group. The $2$-generated group $G$ from Lemma \ref{D(s)} can be presented as $G\cong F_2/N$, where the normal subgroup $N$ contains the second derived subgroup $F_2''$
since $G$ is metabelian. On the other one hand, $N\ne F_2''$ since the quotient $F_2/F_2''$ is an
extension of the free abelian group $F_2'/F_2''$ by the free abelian group $F_2/F_2'$, and so it
cannot  contain an element divisible by all powers $n^k$ for $n\ge 2.$ Observe also
that $N\le F_2'$ since $F_2/N$ maps onto $F_2/F_2'$ by the second statement of Lemma \ref{D(s)}.
Hence $N/F_2''$ is a non-trivial subgroup of the free abelian group $F_2'/F_2''$.

Since the group $F_2'/F_2''$
is free abelian, the intersection $\cap_{m\in I} A_m(F_2'/F_2'')$ is trivial for arbitrary
infinite set $I$ of positive integers, i.e., $\cap_{m\in I} A_m(F_2')=F_2''.$
Denote $N_m = N \cap A_m(F')$. Since
$$
N/N_m \cong N A_m(F') \, / \, A_m(F') \le F' / A_m(F') \in \A_m,
$$
the group $N/N_m$ has finite exponent dividing $m$. Taking large enough powers of a prime $p$ as the values of $m$ we can get elements with infinitely growing orders in the groups $N/N_m$, for, if the orders of all elements of $N/N_m$ for all $m=p^k$ were restricted from above by a single number, then the quotient
$N/(\cap_{k=1}^{\infty} N_{p^k})$ would have a finite exponent, a contradiction with the facts that  $\cap_{k=1}^{\infty} N_{p^k}\le \cap_{k=1}^{\infty} A_{p^k}(F_2')=F_2''$ and $N/F''$ is a nontrivial
torsion free group.
Hence for a given $k\ge 1$, we can take $m$ large enough so that the exponent of the group $N/N_m$  to be
divisible by $p^k.$  Since one can choose such $m_1,..., m_s$ for every prime-power divisor $p^{k_i}$ of
$n$, there is $m=m_1\dots m_s$ such that the abelian group $N/N_m$ has an element of order $n.$

The group $N/N_m$ is isomorphic to the normal subgroup $N A_m(F_2') \, / \, A_m(F_2')$ of $F_2/ A_m(F_2')$. By Lemma \ref{wr} (with $L=F_2'$), both these groups contain subgroups isomorphic with
$Z(n)$.

By the choice of $G$, the abelian derived subgroup of the $2$-generator metabelian group $M =F_2/N_m$  contains
a subgroup $U=R/N_m$ such that $R/N\cong D(n).$  Therefore the subgroup $
T=N/N_m$ contains all torsion elements of $U$, has exponent $\le m$, and by Kulikov's theorem~\cite{Robinson, Fuchs},
$U\cong T\times D(n)$. Since $T$ contains a subgroup $Z(n)$, the group $M'$ contains
a subgroup isomorphic with $Z(n)\times D(n)$, as required.
\end{proof}

{\bf Proof of Theorem \ref{main}} The statements of the theorem follow from Lemmas \ref{nec}, \ref{Z(n) x D(n)}, and Remark \ref{Dn}. $\Box$

\medskip

%%%%%%%%%%%%%%%%%%%%%%%%%%%%%%%%%%%%%%%%%%%%%%%%%%%%%%%%%%%%%%%%%%
%%%%%%%%%%%%%%%%%%%%%%%%%%%%%%%%%%%%%%%%%%%%%%%%%%%%%%%%%%%%%%%%%%
%%%%%%%%%%%%%%%%%%%%%%%%%%%%%%%%%%%%%%%%%%%%%%%%%%%%%%%%%%%%%%%%%%
\section{Non-isomorphic subgroups of $D(p)$}
\label{Non-isomorphic subgroups of D(p)}

The set of non-isomorphic groups $B$
in the class $\mathcal B$ is countable. Indeed Ph. Hall observed~\cite{On the finiteness of certain soluble groups} that each
$B$ is a finitely generated module over the group ring ${\Z}P$ of a polycyclic
group $P$ and he proved that such modules are Noetherian. Thus the set of non-isomorphic
${\Z}P$-modules is countable, and it suffices to take into account that the set
of non-isomorphic polycyclic groups also is countable.

 We want to prove that for every prime $p$, the group $D(p)$ has $2^{\aleph_0}$
 non-isomorphic subgroups, which together with Theorem \ref{main} prove Theorem 2.
 Our proof is based on Chapter XIII of Fuchs' book~\cite{Fuchs}.

The system of torsion free abelian groups $\{ G_i | i \in I \}$ is said to be {\it rigid} if $Hom (G_i,G_i) \le \Q$, and  $Hom (G_i,G_j) \cong 0$ for any $i,j\in I$, $i \not= j$. In other words, each endomorphism of a group in a rigid system is a multiplication by a rational number, and there only is zero homomorphism between two distinct groups of the system.
An example of a rigid system of continuum cardinality is constructed in Example 5, Section 88 \cite{Fuchs}. Given a prime $p$ and $r\ge 2$, take $r-1$  algebraically independent (over rational field $\Q$) $p$-adic units $\pi_2, \ldots, \pi_r$, and take an extra $\pi_1=1$.
Let $\pi_{in}$ be the $n-1$-th partial sum of the canonic presentation of $\pi_i$:
$$
\pi_i = s_{i0}+ s_{i1}p + \cdots + s_{in}p^n + \cdots \quad (0 \le s_{in} < p).
$$
Denote:
$$
x_n= p^{-n}(a_1 + \pi_{2n}a_2 + \cdots + \pi_{rn} a_r),
$$
where $a_1, \ldots, a_r$ is a basis of in the vector space $\Q^r$,
%(for simplicity assume it is the orthonormal basis),
and take the group:
$$
A_{\pi_2, \ldots, \pi_r} = \langle
a_1, \ldots, a_r, \, x_1, \ldots, x_n , \ldots
\rangle.
$$
From the construction of $x_n$ above it is clear that $A_{\pi_2, \ldots, \pi_r}$  is a subgroup in the  power $D_p^r \le D(p)$. As it easily follows from the properties of $p$-adic numbers, taking a bigger set of algebraically independent $p$-adic units $\pi_2,\dots,\pi_s, \pi'_2, \ldots, \pi'_r$ we get that  $Hom(A_{\pi_2, \ldots, \pi_r} , A_{\pi'_2, \ldots, \pi'_r}) = 0$ (see Example 5, Section 88 \cite{Fuchs}). In particular, since there are continuously many algebraically independent $p$-adic units, we have $2^{\aleph_0}$  pairwise non-isomorphic subgroups in $D_p^r$.

Thus Theorem~\ref{cardin} is proved: $2^{\aleph_0}$ pairwise non-isomorphic subgroups $A_{\pi_2, \ldots, \pi_r}$ are contained in $D(p)$, and the latter has an isomorphic embedding in a 2-generator metabelian group by Lemma~\ref{Z(n) x D(n)}.

\vskip3mm
Theorem~\ref{cardin} also means that the classes of  ${\mathcal B}$-groups and of Hall groups are different, although there is similarity in their definitions (a ${\mathcal B}$-group is the Hall group, which is not only embeddable into a finitely generated abelian-by-polycyclic group, but also has a normal embedding in such a group). In particular, it is easy to see that these two classes coincide over torsion groups.
%Observe that these two classes coincide over periodic abelian groups: a periodic ${\mathcal H}$-group is a ${\mathcal %B}$-group. For, a periodic ${\mathcal H}$-group $H$ has finite period, and by Kulikov's theorem it can be presented as %a direct product of finitely many direct summands, each of which is a finite or countable direct power of a finite %cycle: $H = H_1 \oplus \cdots \oplus H_n$, with $H_i = \bigoplus_{j \in J_i} \Z_{k_i}$, $i=1,\ldots, n$.
%
%Each of these $H_i$ is a normal subgroup (and actually is the base subgroup $\Z_{k_i}^{\Z_{J_i}}$) in direct wreath %product $\Z_{k_i} wr \Z_{J_i}$, where $\Z_{J_i}$ is a finite or infinite cycle, the order of which is equal to (finite %or countable) cardinality of the set $J_i$.
%
%Thus $H$ can normally be embedded into the finitely generated metabelian group  $\sum_{i=1}^{n}\Z_{k_i} wr \Z_{J_i}$.
Let us conclude this section by an explicit example of a Hall group which is not a ${\mathcal B}$-group.

\begin{Example}
Let $G$ be the group built in 4.4.2 in~\cite{Robinson}. $G$ is an indecomposable  abelian group of rank two, and is defined as a subgroup in 2-dimensional vector space $\Q^2$ with basis $u$, $v$ as follows: for three distinct primes $p, q, r$ set $G$ to be the subgroup of $\Q^2$ generated by all elements:
$$
p^m u, \quad q^m, \quad r^m(u+v)
$$
for all integer values $m$. Evidently $G$ is a Hall group, since it is subgroup in $D_{pqr}^2$.

Assume that $G\in \mathcal B$, i.e., we are able to normally embed $G$ into an abelian-by polycyclic finitely generated group $M$. Then the centralizer
$C=C_M(G)$ contains the abelian group $G$ and has index $[M:C]\le 2$ since $M/C$ faithfully acts  on $G$ by conjugation automorphisms, and
the only automorphisms of $G$ are $\pm id_G$ (see 4.4.2 in~\cite{Robinson}). Therefore $C$ is also a finitely
generated abelian-by polycyclic group, and by ~\cite{On the finiteness of certain soluble groups},
$G$ is a finitely generated $C$-module. Since the action of $C$ on $G$ is trivial, $G$ has to be a finitely
generated abelian group, a contradiction.
%and for simplicity of notation identify $G$ with its image in $M$.% Since $G \normal M$, conjugations of $G$ by %elements of $M$ are automorphisms. This defines an exact representation by automorphisms in $\Aut G$ of the group %$M/C$, where $C = C_M(G) $ is the centralizer of $G$ in $M$ (to each coset of $M$ by $C$ corresponds the automorphism %in $\Aut G$ given by conjugation of $G$ by an element from that coset).
%
%As it is proved in 4.4.2 in~\cite{Robinson}, the only automorphisms of $G$ are $\pm id_G$.
%has two automorphisms only: the trivial automorphism and the automorphism mapping each vector $w$ to $-w$. Thus, $\Aut %G \cong \Z_2$, and $M/C$ either is trivial or is isomorphic to $\Z_2$.
%
%Since $G$ is an abelian group, it is contained in centralizer $C$, and $C$ also is a normal subgroup in $M$. We can at %once rule out the case when $M$ is abelian, since $G$ cannot be contained in a finitely generated abelian group.
%
%$C$ also is finitely generated as a subgroup of finite index in a finitely generated group. Since $C/G$ is %polycyclic,
%$G$ is a finitely generated $\Z(C/G)$ by~\cite{On the finiteness of certain soluble groups}. Since $G$ is central in %$C$, this module is over $\Z(C/C) = \Z$, that is, $G$ is also is  finitely generated just as an abelian group. This %contradiction completes the proof.
\end{Example}

%%%%%%%%%%%%%%%%%%%%%%%%%%%%%%%%%%%%%%%%%%%%%%%%%%%%%%%%%%%%%%%%%%
%%%%%%%%%%%%%%%%%%%%%%%%%%%%%%%%%%%%%%%%%%%%%%%%%%%%%%%%%%%%%%%%%%
%%%%%%%%%%%%%%%%%%%%%%%%%%%%%%%%%%%%%%%%%%%%%%%%%%%%%%%%%%%%%%%%%%
\section{Embeddings of finitely generated abelian groups}
\label{Embeddings of finitely generated abelian groups}

\begin{Proposition} Every finitely generated abelian group $H$ is isomorphic to
the center  of a $2$-generated metabelian, nilpotent group $G$. If $H$ is finite
then so is $G$.
 \end{Proposition}

 \begin{proof} Let $F$ be a free group of free rank $2$ in the variety of metabelian
 and nilpotent of class $\le c$ groups $(c\ge 2).$ The $c$-th member $C$ of
 the lower central series of $F$ is contained in the center of $F$. Furthermore, by Corollary 36.23 of \cite{NeumannHNeumann}, $C$ is a free abelian
 group of rank $c-1$.

 Given an $n$-generated abelian group $H$, it can be presented as a quotient $C/B$
 if $c-1\ge n,$ and so the group $H$ is isomorphic to a central factor $L/K$ of $F$. Since a finitely generated nilpotent group satisfies the ascending chain condition for subgroups~\cite{Robinson}, we can assume that $L$ is a maximal
  normal subgroup of $F$ such that $H\cong L/K$ for some normal subgroups $K$ of $F$ and $L/K$ is
  a central subgroup in $F/K.$

 Now we want to prove that $L/K$ is the center of $G=F/K$. Arguing by contradiction, we have
 a bigger central subgroup $A=M/K$ in $F/K$, i.e. $M>L$. But every subgroup of a finitely generated abelian
 group $A$ is isomorphic to  a factor group of $A$ (trivially follows from~\cite[Theorem 15.6]{Fuchs1}). Hence the subgroup $H\cong L/K\le A$ is isomorphic to a
 central factor $M/N$ of $F$, contrary to the maximality of $L$.

 If $H$ is finite then so is $G$ since the center of an infinite finitely generated, nilpotent
 group is infinite~\cite[Exercise 17.2.10]{KargapolovMerzlyakov}. The theorem is proved.
 %
%Every subgroup of a finitely generated abelian group is isomorphic to a factor group of it. Thus, if $H < Z$ holds, %then $H \cong Z/N$ for a normal subgroup $N \not\cong Z$ of $Z$. The subgroup $N$ also is normal (and central)  in %$M$, so the factor group $M/N$ is correctly defined and it has $Z/N$ as a central subgroup. Let $Z_1/N$ be the center %of $M/N$ (with $Z_1$ being the full pre-image of the center of $M/N$ under natural epimorphism). If $Z < Z_1$, we can %repeat the step and present $H$ as the factor group $Z_1/N_1$ for some $N_1>N$, since $Z_1$ also is finitely %generated. This way we get the chain $Z < Z_1< Z_2< \cdots$ of normal subgroups in $M$. Since $M$ satisfies max-n, we %get $Z_{k}= Z_{k+1}$ for some $k$, and $H \cong M/Z_{k+1}$.
\end{proof}

\begin{Proposition} Every finitely generated abelian group $H$ is a normal subgroup
with a finite cyclic quotient $G/H$ in a $2$-generated metabelian group $G$.
 \end{Proposition}
 \begin{proof}
 The group $H$ is the direct sum of cyclic subgroup
$$
\langle z_1 \rangle \oplus\dots\oplus \langle z_m \rangle \oplus \langle z_{m+1} \rangle \oplus
\dots\oplus \langle z_{m+n} \rangle \quad (m,n\ge 0),
$$
where
 $z_1,\dots, z_m$ have infinite orders, and the order $n_{i+1}$ of $z_{i+1}$ divides the finite order
 $n_i$ of $z_i$ for $i=m+1,\dots, m+n-1.$

Then mapping $\varphi$ on the generators of $H$ given by the rules
\begin{align}
\label{df}
  &\begin{aligned}
   & z_1\mapsto z_2,\; z_2\mapsto z_3, \dots, z_{m-1}\mapsto z_m, \\
   &  z_m\mapsto z_1 z_{m+1}, z_{m+1} \mapsto z_{m+1} z_{m+2}, \dots, z_{m+n-1} \mapsto z_{m+n-1} z_{m+n}, \\
   & z_{m+n} \mapsto z_{m+n}
  \end{aligned}
\end{align}
extends to the endomorphism of $H$ since it is well defined on the direct summand
due to the conditions $n_{i+1}|n_i$. This endomorphism (denoted also by $\varphi$) is surjective
since it is surjective on the $\varphi$-invariant subgroup $L=\langle z_{m+1},\dots, z_{m+n} \rangle$ and
on the quotient $H/L.$ By the Hopf property of $H$, $\varphi$ is its automorphism.

Furthermore, $\varphi^m$ induces the identity automorphism of $H/L$, and $\varphi^{ms}$
is identical on $L$ for some $s>0$ since the subgroup $L$ is finite. It follows that
$\varphi^{msl} = id_H$, where $l$ is the order of $L$.

Let now $G$ be the semidirect product of $H$ and the finite cyclic group $\langle \varphi \rangle$
with the defined above action of $\varphi$. Applying the conjugations by the powers of
$\varphi$  to $z_1$, we successively obtain from (\ref{df}) that $z_2,\dots, z_m,z_{m+1},\dots,z_{m+n}\in
\langle \varphi, z_1 \rangle $, that is, $G$ is a 2-generated group, as desired. \end{proof}

%(notice that the last generator only is not moved by $\varphi$, the first $m-1$ generators are shifted %by $\varphi$, and the $m$'th generator is mapped to the product of elements of infinite and finite %order).

%%%%%%%%%%%%%%%%%%%%%%%%%%%%%%%%%%%%%%%%%%%%%%%%%%%%%%%%%%%%%%%%%%
%%%%%%%%%%%%%%%%%%%%%%%%%%%%%%%%%%%%%%%%%%%%%%%%%%%%%%%%%%%%%%%%%%
%%%%%%%%%%%%%%%%%%%%%%%%%%%%%%%%%%%%%%%%%%%%%%%%%%%%%%%%%%%%%%%%%%
\section{Further discussion}
\label{Further discussion}

The fact that countable abelian, nilpotent, or generalized nilpotent groups are not in general embeddable into finitely generated groups from the
mentioned classes, is a background of the embedding theorems for countable solvable groups: solvability, in some sense, is the ``first property'' that can be added to embeddings of countable groups into finitely-generated groups.

One of the main motives in theory of embeddings of groups is: Given an embedding of the group $H$ into the group $G$, then ``how close'' is the group $G$ to $H$? By the well-known result of B.H.~Neumann and Hanna Neumann~\cite{NeumannHNeumann}, every countable solvable group of solvability length $l$ is embeddable into a $2$-generated solvable group of length $l+2$  but not, in general, of $l+1$. To show that $l+2$ may not be replaced by  $l+1,$ Ph. Hall \cite{Hall on Finiteness Conditions for soluble groups} (Lemma 2) and B.H.~Neumann and Hanna Neumann ~\cite{NeumannHNeumann} (Lemma 5.3) bring explicit examples for $l=1$, namely the additive group of rational numbers $\Q$ and the quasicyclic groups $C(p^\infty)$, resp.
%Theorem~\ref{main} of the current work provides wide generalization of Lemma~5.3 and classifies all the abelian groups %that can be embedded into finitely-generated (2-generated) metabelian groups. And Remark~\ref{Dn} provides an %easy-to-memorize description of such groups: an abelian group is embeddable into a finitely-generated (2-generated) %metabelian group if and only if it is a subgroup of the infinite (countable) direct power of the group $\Z_n \oplus %D_n$ for some positive integer $n$.
That lemmas can  be generalized in one more direction: counter-examples of the above mentioned type can be found for any $l\ge 1$.

\begin{Example}
Denote by $H$ the  group of all upper unitriangular $n\times n$ matrices over $\Q$.
 It is easy to check that $H$ is a (uniquely) divisible group, and
it is solvable of prescribed length $l\ge 1$ for $n=2^{l-1}+1$.
%an appropriate $n=n(l)$ (in fact, $l=\lceil\log_2 n\rceil$).
The group $H$ does not embed
into a finitely generated group $G$ of solvable length $l+1$. Indeed, the finitely generated
abelian group $G/G'$ has no non-trivial divisible subgroups, i.e., $HG'/G'$ is trivial, and so $H\le G'$.
%Moreover, $H\le G''$ since
Note that no non-trivial subgroup of $G'/G''$  is divisible because $G'/G''$ is  a Hall group  by \cite{Hall on Finiteness Conditions for soluble groups}. Hence $HG''/G''$ is trivial, i.e., $H\le G''$, and we come
to a contradiction since the length of $G''$ is at most $l-1.$
\end{Example}

\begin{Problem}  For any $l\ge 2$ obtain  an embedding criterion (similar to that given in Theorem \ref{main} for $l=1$): Which countable groups of solvable length $l$ are embeddable in finitely
generated solvable groups of length $l+1$ ?
\end{Problem}

If in case $l=1$, one replaces metabelian groups by a slightly larger class of center-by-metabelian
groups, then all the restrictions on $H$ will be removed, since Ph. Hall~\cite[Theorem 6]{Hall on Finiteness Conditions for soluble groups} showed that every countable abelian  group $H$ is the center of a $2$-generated center-by-metabelian group. On the other
hand, the examples of Section~\ref{Embeddings of finitely generated abelian groups} demonstrate that finitely generated abelian groups admit embeddings
in $2$-generated metabelian groups with nice additional properties. What can we expect if $l\ge 2$?
The sharper formulation:

\begin{Problem}  Does every finitely generated solvable group of length $l\ge 2$ embed
into a $2$-generated solvable group of length $l+1$ ? Or at least, into some
$k$-generated $(l+1)$-solvable group, where $k=k(l)$ ?
\end{Problem}

{\bf Acknowledgments.} We would like to thank Avinoam Mann and Vitaly Roman'kov for their comments. 
The first author is happy to acknowledge very warm hospitality he enjoyed during his research visit to the Vanderbilt University. The second author was supported in part by the NSF grants DMS-1161294 and by the RFBR grant 11-01-00945.

%%%%%%%%%%%%%%%%%%%%%%%%%%%%%%%%%%%%%%%%%%%%%%%%%%%%%%%%%%%%%%%%%%
%%%%%%%%%%%%%%%%%%%%%%%%%%%%%%%%%%%%%%%%%%%%%%%%%%%%%%%%%%%%%%%%%%
%%%%%%%%%%%%%%%%%%%%%%%%%%%%%%%%%%%%%%%%%%%%%%%%%%%%%%%%%%%%%%%%%%

%%%%%%%%%%%%%%%%%%%%%%%%%%%%%%%%%%%%%%%%%%%%%%%%%%%%%%%%%%%%%%%%%%
%%%%%%%%%%%%%%%%%%%%%%%%%%%%%%%%%%%%%%%%%%%%%%%%%%%%%%%%%%%%%%%%%%
%%%%%%%%%%%%%%%%%%%%%%%%%%%%%%%%%%%%%%%%%%%%%%%%%%%%%%%%%%%%%%%%%%

\vskip11mm
{\small
\begin{tabular}{p{2.3in}p{3in}}
Vahagn H. Mikaelian:\newline
Department of Applied Mathematics\newline
Yerevan State University\newline
Yerevan 0025, Armenia.\newline
E-mail: v.mikaelian@gmail.com
&
Alexander Yu. Olshanskii:\newline
Department of Mathematics\newline
Vanderbilt University\newline
Nashville, TN 37240, USA.\newline
%\newline
%Department of Mathematics\newline
%Moscow State University\newline
%Moscow 119899, Russia.\newline
E-mail: alexander.olshanskiy@vanderbilt.edu\\
\end{tabular}
}

\end{document}